\documentclass[reqno,11pt]{amsart}
\usepackage[foot]{amsaddr}
\usepackage{bbold}
\usepackage{charter}
\usepackage[margin=1in]{geometry}
\usepackage[colorlinks=true,linktoc=all,linkcolor=blue,citecolor=blue]{hyperref}

\theoremstyle{plain}
\newtheorem{theorem}{Theorem}[section]
\newtheorem*{theorem*}{Theorem}
\newtheorem{proposition}[theorem]{Proposition}
\newtheorem{lemma}[theorem]{Lemma}
\newtheorem{corollary}[theorem]{Corollary}
\theoremstyle{definition}

\newtheorem{definition}[theorem]{Definition}
\newtheorem*{definition*}{Definition}
\newtheorem{remark}[theorem]{Remark}
\newtheorem{example}[theorem]{Example}
\newtheorem{conjecture}[theorem]{Conjecture}
\numberwithin{equation}{section}
\everymath{\displaystyle}
\allowdisplaybreaks
\DeclareRobustCommand*\cal{\relax\mathcal}

\def\n{\noindent}
\def\de{\displaystyle}

\def\le{\leq}
\def\ge{\geq}

\def\B{\mathbb B}
\def\BigC{\mathbb C}

\def\BigN{\mathbb N}

\def\BigD{\mathbb D}
\def\BigB{\mathbb B}

\def\C{\BigC}
\def\bd S{\cal C}

\def\a{\alpha}
\def\b{\beta}

\def\f{\varphi}
\def\g{\gamma}

\def\l{\lambda}
\def\m{\mu}
\def\n{\nu}

\def\r{\rho}
\def\s{\sigma}
\def\t{\tau}

\def\z{\zeta}

\def\G{\Gamma}

\def\>{\geq}
\def\Aut{\operatorname{Aut}}

\def\bd{\partial}

\newcommand{\norm}[1]{\left|\left|#1\right|\right|}
\newcommand{\modu}[1]{\left|#1\right|}

\title{On the Isometric Composition Operators on the Bloch Space in $\mathbb{C}^n$}

\author{Robert F.~Allen and Flavia Colonna}

\address{Department of Mathematical Sciences, George Mason University}
\email{rallen2@gmu.edu, fcolonna@gmu.edu}

\date{}


\begin{document}

\begin{abstract} 	
	Let $\f$ be a holomorphic self-map of a bounded homogeneous domain $D$ in $\BigC^n$. In this work, we show that the composition operator $C_\f: f\mapsto f\circ \f$ is bounded on the Bloch space $\cal{B}$ of the domain and provide estimates on its operator norm. We also give a sufficient condition for $\f$ to induce an isometry on $\cal{B}$. This condition allows us to construct non-trivial examples of isometric composition operators in the case when $D$ has the unit disk as a factor. We then obtain some necessary conditions for $C_\f$ to be an isometry on $\cal{B}$ when $D$ is a Cartan classical domain. Finally, we give the complete description of the spectrum of the isometric composition operators in the case of the unit disk and for a wide class of symbols on the polydisk.
\end{abstract}

\maketitle

\section{Introduction}

An analytic function $f$ on $\BigD=\{z\in \BigC: |z|<1\}$ is said to be {\bf Bloch} if

\begin{equation}\b_f = \sup_{z \in \BigD}(1-|z|^2)|f'(z)| < \infty.\label{blochone}\end{equation}

The set $\cal{B}$ of Bloch functions on $\BigD$ is a Banach space with semi-norm $f\mapsto \b_f$ and norm  $\|f\|_{\cal{B}}=|f(0)|+\b_f$.

By the Schwarz-Pick lemma, if $f$ is any Bloch function on $\BigD$ and $\f$ is an analytic function mapping $\BigD$ into itself, then $f\circ \f\in\cal{B}$ and $\b_{f\circ \f}\le \b_f$, with equality holding if $\f\in{\rm{Aut}}(\BigD)$, where Aut$(D)$ denotes the set of biholomorphic transformations of a domain $D$ (which we call {\bf automorphisms} of $D$).

In \cite{X}, Xiong proved that the composition operator
$C_\f(f)=f\circ \f$ is bounded on $\cal{B}$ and gave estimates for
its operator norm. Furthermore, he obtained several necessary
conditions for $C_\f$ to be an isometry (that is, to preserve the
Bloch norm).  The function $\f$ is called the {\bf symbol} of the
operator $C_\f$.

In \cite{C}, the second author completely characterized the symbols of the isometric composition operators. In Theorem~\ref{cmain}, we list several equivalent conditions that can be used to describe the class of such symbols.

\begin{theorem} \label{cmain} Let $\f$ be an analytic function mapping $\BigD$ into itself. Then  $C_\f$ is an isometry on $\cal{B}$ if and only if $\f(0)=0$ and any of the following  equivalent conditions holds:
\begin{enumerate}
\item[{\normalfont{(a)}}] $\b_\f=1$.

\item[{\normalfont{(b)}}] $B_\f:=\displaystyle\sup_{z\in \BigD}\frac{(1-|z|^2)|\f'(z)|}{1-|\f(z)|^2}=1$.

\item[{\normalfont{(c)}}] Either $\f\in\hbox{Aut}(\BigD)$ or for every $a\in \BigD$ there
exists a sequence $\{z_k\}$ in $\BigD$ such that $|z_k|\to 1$,
$\f(z_k)= a$, and $$\lim\limits_{k\to\infty}\frac{(1-|z_k|^2)|\f'(z_k)|}{1-|\f(z_k)|^2}=1.$$

\item[{\normalfont{(d)}}] Either $\f\in\hbox{Aut}(\BigD)$ or for every $a\in \BigD$ there
exists a sequence $\{z_k\}$ in $\BigD$ such that $|z_k|\to 1$,
$\f(z_k)\to a$, and $$\lim\limits_{k\to\infty}\frac{(1-|z_k|^2)|\f'(z_k)|}{1-|\f(z_k)|^2}=1.$$

\item[{\normalfont{(e)}}] Either $\f\in\hbox{Aut}(\BigD)$ or the zeros of $\f$ form an
infinite sequence $\{z_k\}$ in $\BigD$ such that
$$\limsup_{k\to\infty}(1-|z_k|^2)|\f'(z_k)|=1.$$

\item[{\normalfont{(f)}}] Either $\f\in\hbox{Aut}(\BigD)$ or $\f=gB$, where $g$ is a
non-vanishing analytic function mapping $\BigD$ into itself or a
constant of modulus 1, and $B$ is an infinite Blaschke product whose
zero set $Z$ contains a sequence $\{z_k\}$ such that $|g(z_k)|\to 1$
and $$\lim_{k\to\infty} \prod_{\zeta\in Z, \zeta\ne
z_k}\left|\frac{z_{k}-\z}{1-\overline{z_{k}}\z}\right|=1.$$

\item[{\normalfont{(g)}}] Either $\f\in\hbox{Aut}(\BigD)$ or there exists
$\{S_k\}_{k\in\BigN}$ in $\hbox{Aut}(\BigD)$ such that $|S_k(0)|\to
1$ and $\{\f\circ S_k\}$ approaches the identity locally uniformly
in $\BigD$.
\end{enumerate}
\end{theorem}

Condition (d) was noted in \cite{MV}. For the equivalence of the other conditions see \cite{CC2}, \cite{CC3}, \cite{C}, and \cite{C3}.

Condition (f) yields a recipe for constructing examples of isometric composition operators induced by symbols that are not rotations. Easily constructible examples are Blaschke products whose zero sets $\{z_k\}$ are {\bf thin}, that is, satisfy the condition $$\lim_{k\to\infty} \prod_{j\ne k}\left|\frac{z_{k}-z_j}{1-\overline{z_{k}}z_j}\right|=1.$$

In this article, we study the symbols of the isometric composition operators on the Bloch space of a bounded homogeneous domain in $\C^n$ ($n\in \BigN$, $n\ge 2$) and, when needed, restrict our attention to bounded symmetric domains. Our goal is obtain in higher dimensions conditions analogous to (a)-(g) in Theorem~\ref{cmain} to whatever extent is possible.

\subsection{Organization of the paper}

In Section 2, we review the background on Bloch functions on a bounded homogeneous domain in $\BigC^n$. Furthermore, we recall the Cartan classification of bounded symmetric domains and the corresponding invariant metrics, known as {\bf Bergman metrics}.

In Section 3, we give a characterization of the Bloch semi-norm of a Bloch function on a bounded homogeneous domain as the maximum dilation of $f$ with respect to the Bergman distance, which allows us to show that the composition operator $C_\f$ is bounded on the Bloch space $\cal{B}$ of the domain. We also provide estimates for its operator norm. The boundedness of $C_\f$ was observed in \cite{SL} and in \cite{ZS}, but the authors only showed boundedness with respect to the Bloch semi-norm (which yields the boundedness with respect to the Bloch norm under the assumption that $\f$ fixes the base point of the domain used to define the Bloch norm). We then show that the correspondence $f\mapsto \b_f$ is lower-semi-continuous on the Bloch space.

In Section 4, we recall the classification of the Bloch constant of a bounded symmetric domain $D$ (namely, the supremum of the Bloch semi-norms of the holomorphic functions mapping $D$ into $\BigD$) in terms of the Cartan class and the dimension of $D$. We then show that for most domains, the canonical projections map $D$ into $\BigD$ and have maximal Bloch semi-norm, a property that will be used in Section~6 to derive a necessary condition for a holomorphic self-map of a bounded symmetric domain of this type to induce an isometry on the Bloch space.

In Section 5, we give some simple properties of the maximum Bergman dilation of a holomorphic self-map and use the lower-semi-continuity of the Bloch semi-norm to obtain a sufficient condition for a holomorphic self-map on a bounded homogeneous domain $D$ to induce an isometry on $\cal{B}$ under composition. We use this result to generate non-trivial examples in the special case when $D$ has $\BigD$ as a factor.

In Section 6, we give some necessary conditions for a holomorphic self-map of $D$ to be the symbol of an isometric composition operator on $\cal{B}$ when $D$ is a product of Cartan classical domains. We conclude the section with a conjecture. 

In Section 7 we analyze the spectrum of $C_\f$ when this is an isometry on $\cal{B}$ in the case when $D$ is the unit polydisk. We deduce the full description of the spectrum in the one-dimensional case.

We conclude the paper (Section 8) with a list of open questions.

\section{Bloch functions in higher dimensions}

A {\bf homogeneous domain} in $\BigC^n$ is a domain $D$ such that Aut$(D)$ acts transitively on $D$, i.e. given any two points $z,w\in D$ there exists $T\in {\rm{Aut}}(D)$ such that $T(z)=w$.

In \cite{H}, Hahn introduced the notion of Bloch function on a
bounded homogeneous domain in $\BigC^n$.

Let $f$ be a complex-valued holomorphic function on a bounded homogeneous domain $D$ in $\BigC^n$.
For $u,v\in \C^n$, let $\langle u,v\rangle =\de\sum_{k=1}^nu_k\overline{v_k}$, and for $z\in D$, let $(\nabla f)(z)u=\langle (\nabla f)(z), \overline{u}\rangle$, where $(\nabla f)(z)$ is the gradient of $f$ at $z$. Denote by $H_z(u,\overline{v})$ the Bergman metric on $D$, that is, the positive definite Hermitian form which is invariant under automorphisms of $D$. Using a different approach from Hahn's, in \cite{T1} Timoney defined a Bloch function on $D$ as a holomorphic function $f$ on $D$ such that
$\b_f=\sup\limits_{z\in D}Q_f(z)$ is finite, where $$Q_f(z)=\sup_{u\in \BigC^n\backslash\{0\}}\frac{|(\nabla f)(z)u|}{H_z(u,\overline{u})^{1/2}}.$$ Denote by $\cal{B}$ the space of Bloch functions on $D$. The map $f\mapsto \b_f$ is a semi-norm on $\cal{B}$, and fixing any point $z_0\in D$, the set $\cal{B}$ is a Banach space, called the {\bf Bloch space}, under the norm $\|f\|_{\cal{B}}=|f(z_0)|+\b_f$.
For convenience, throughout this paper we shall assume that $0\in D$ and $z_0=0$.  The Bloch space contains the space of bounded holomorphic functions \cite{T1}.  When $D=\BigD$, the above definition of Bloch function
reduces to condition (\ref{blochone}).

Excellent references on Bloch functions include \cite{ACP} for the one-variable case, \cite{T1} and \cite{T2} for the several-variable case. See also \cite{Ci} for an overview of Bloch functions on the disk and its connection to other function spaces.  Bloch functions have been defined on more general classes of bounded domains, such as strongly pseudo-convex domains \cite{KM}. These domains, however, have sparse, and often trivial, automorphism groups, and operator theory problems are much more difficult to treat.

A domain $D$ in $\BigC^n$ is said to be {\bf symmetric} if for each $z_0\in D$, there exists an automorphism $S$ of $D$ which is an involution (i.e. $S\circ S$ is the identity) and has $z_0$ as an isolated fixed point. Symmetric domains are homogeneous (see \cite{H}, pp. 170, 301). Examples of symmetric domains are the unit ball
$$\BigB_n=\{z=(z_1,\dots,z_n)\in \BigC^n: \|z\|<1\},$$ where $\|z\|=\left(\sum_{k=1}^n|z_k|^2\right)^{1/2}$, and the unit polydisk $$\BigD^n=\{z=(z_1,\dots,z_n)\in \BigC^n: |z_k|<1 \hbox{ for all }k=1,\dots,n\}.$$
We now describe their groups of automorphisms, and their Bergman metrics.

The automorphism group of the unit ball can be described as
$$\Aut(\BigB_n)=\{U\circ \f_a: U \hbox{ unitary}, a\in \BigB_n\},$$
where for $z\in \BigB_n$, $\f_0(z)=z$, while for
$a\in\BigB_n\backslash\{0\}$,
\begin{equation}\f_a(z)=\frac{a-P_a(z)-(1-\|a\|^2)^{1/2}Q_a(z)}{1-\langle
z,a\rangle},\label{fia}\end{equation} with    $P_a(z)=\frac{\langle
z,a\rangle}{\|a\|^2}a$, and $Q_a(z)=z-P_a(z)$ (see \cite{R},
Theorem~2.2.5).  For $z\in\B_n$, $u,v\in\BigC^n$, the Bergman metric
on $\B_n$ is given by
$$H_z(u,\overline{v})=\frac{(n+1)[(1-\|z\|^2)\langle u,v\rangle+\langle u,z\rangle \langle z,v\rangle]}{2(1-\|z\|^2)^2}.$$

The automorphism group of the polydisk is given by $$\Aut(\BigD^n)=\{(T_1(z_{\t(1)}), \dots, T_n(z_{\t(n)})): T_k\in \Aut(\BigD),\t \in S_n\},$$ where $S_n$ is the group of permutations on the set $\{1,\dots,n\}$ (see \cite{R2}, p. 167).
For $z\in \BigD^n$, $u,v\in\BigC^n$, the Bergman metric on $\BigD^n$
is given by
$$H_z(u,\overline{v})=\sum_{k=1}^n\frac{u_k\overline{v_k}}{(1-|z_k|^2)^2}.$$

For a function $\f$ mapping a domain $D\subset \BigC^n$ into $\BigC^n$ and for $z\in D$, let $J\f(z)$ be the Jacobian matrix of $\f$ at $z$, that is, the $n\times n$ matrix whose $(j,k)$-entry is $\partial \f_j(z)/ \partial z_k$.  By definition of the Bergman metric, if $D$ is a bounded homogeneous
domain and $S\in \hbox{Aut}(D)$, then for all $z\in D$, and for all
$u\in \BigC^n$ \begin{equation}
H_{S(z)}(JS(z)u,\overline{JS(z)u})=H_z(u,\overline{u}),\label{equa}\end{equation}
where $JS(z)u$ is the usual matrix product where $u$ is viewed as a
column vector.  As an immediate consequence of (\ref{equa}), we
deduce that the Bloch semi-norm is invariant under right composition
of automorphisms.

We use the notation $H(D,D')$ for the class of holomorphic functions mapping a domain $D$ in $\BigC^n$ into a domain $D'$ in $\BigC^m$ and use the abbreviation $H(D)$ for $H(D,D)$.

In \cite{T1} (proof of Theorem 2.12) it was shown that for $n,m\in \BigN$, if $D\subset \BigC^{n}$ and $D'\subset \BigC^{m}$ are bounded homogeneous domains, then there is a constant $c>0$ depending only on $D$ and $D'$ such that for any $\f\in H(D,D')$, $z \in D$, $u \in \mathbb{C}^n$
$$H^{D'}_{\f(z)}(J\f(z)u,\overline{J\f(z)u})\le c H^{D}_z(u,\overline{u}),$$
where $H^{D}_z$ and $H^{D'}_{\f(z)}$ are the Bergman metrics on $D$ and $D'$ at $z$ and $\f(z)$, respectively.
In particular, if $D$ is a bounded homogeneous domain and $\f$ is a holomorphic map of $D$ into itself, then
$$B_\f=\sup_{z\in D}\sup_{u\in\BigC^n\backslash\{0\}}\frac{H_{\f(z)}(J\f(z)u,\overline{J\f(z)u})^{1/2}}{H_z(u,\overline{u})^{1/2}}$$
is bounded above by a constant independent of $\f$. Thus, for any Bloch function $f$ on $D$ and any $z\in D$, we have
$$Q_{f\circ \f}(z)\le \sup_{u\in \BigC^n\backslash\{0\}}\left(\frac{H_{\f(z)}(J\f(z)u,\overline{J\f(z)u})}{H_z(u,\overline{u})}\right)^{1/2} Q_f(\f(z)).$$ Thus
\begin{equation} Q_{f\circ \f}(z)\le B_\f Q_f(\f(z)).\label{mainf}\end{equation}
 Furthermore,
$Q_{f\circ \f}(z)=Q_f(\f(z))$ if $\f\in{\rm Aut}(D)$. Consequently,
\begin{equation} \b_{f\circ \f}\le B_\f\b_f,\label{semibdd}\end{equation}
and so composition operators whose symbol fixes $0$ are bounded on $\cal{B}$ and the composition operators induced by the automorphisms of $D$ preserve the Bloch semi-norm. Moreover, if $C_\f$ preserves the Bloch semi-norm, then $B_\f\ge 1$. In Section 3, we shall prove that $C_\f$ is in fact bounded on the Bloch space without the assumption $\f(0)=0$.

We call $B_\f$ the {\bf{Bergman constant}} of $\f$. By (\ref{equa}), the Bergman constant of an automorphism is 1. When $n=1$, $D=\BigD$, and $\f$ is any analytic self-map of $\BigD$ we see that $H_z(u,\overline{u})=\frac{|u|^2}{(1-|z|^2)^2}$. Thus, $B_\f$ is the constant in part (b) of Theorem~\ref{cmain}.

The one-dimensional case suggests that the value of the Bergman constant may play a role in helping us determine which holomorphic self-maps of a bounded homogeneous domain induce composition operators that are isometries on the Bloch space. The higher-dimensional case is more difficult to treat because, in general, there may exist functions $\f\in H(D)$ whose Bergman constant is larger than 1. Indeed, in \cite{Ko} Kor\'anyi proved that the value of the maximal Bergman constant is related to the rank of the domain. We state his result in  Theorem~\ref{koranyi}. An example of a function $\f$ whose Bergman constant is larger than 1 is $\f(z)=(z_1,\dots,z_1)$ for $z\in\BigD^n$. 
However, the corresponding operator $C_\f$ is not an isometry on
$\cal{B}$. Indeed, as we shall observe in Theorem~\ref{neccond}, if
$C_\f$ is an isometry on $\cal{B}$, then the components of $\f$ are
linearly independent.

Because of the invariance of the semi-norm under composition by automorphisms, easy examples of symbols of isometric composition operators are the automorphisms that fix 0. In the case of the unit polydisk they are the functions $(z_1,\dots,z_n)\mapsto (\l_1z_{\t(1)},\dots,\l_nz_{\t(n)})$ where the $\l_j$ ($j=1,\dots,n$) are unimodular constants and $\t\in S_n$. In the case of the unit ball, they are the unitary transformations.

In \cite{CC3} the problem of characterizing the holomorphic self-maps of the unit polydisk which are symbols of isometric composition operators was studied. Although a complete characterization was not achieved, several necessary conditions were obtained and a large class of symbols of isometric composition operators was described. Other references on composition operators on the Bloch space in higher dimensions include \cite{SL} and \cite{ZS}.

Cartan \cite{C1} proved that any bounded symmetric domain is biholomorphically equivalent to a finite product of irreducible bounded symmetric domains, unique up to rearrangement of the factors. He then classified all the irreducible domains into six classes, four of which form infinite families (known as {\bf{ Cartan classical domains}}) and two classes each containing a single domain of dimension 16 and 27, respectively, called {\bf{exceptional domains}}. All these classes contain the origin. A bounded symmetric domain $D$ is said to be in {\bf standard form} if it has the form $D=D_1\times \cdots\times D_k$, where each $D_j$ is a Cartan classical domain or an exceptional domain.

The Cartan classical domains $R_I,R_{II},R_{III},R_{IV}$ and the
respective Bergman metrics are described below. The notation we
use conforms to Kobayashi's in \cite{Kob}, except we scale the
Bergman metrics by dividing them by 4.  For a description of their
automorphism groups see \cite{Kob}. For a description of the
exceptional domains $R_{V}$ and $R_{VI}$, see \cite{D}.

Let $\cal{M}_{m\times n}$ denote the set of $m\times n$ matrices over $\BigC$ and let $\cal{M}_n=\cal{M}_{n\times n}$. Let $I_n\in\cal{M}_n$ be the identity matrix and let $Z^*$ be the adjoint of the matrix $Z$. Then
$$\begin{aligned} R_I&=\{Z\in\cal{M}_{m\times n}: I_m-ZZ^*>0\},\hbox{ for }m\ge n\ge 1,\\
H_Z(U,\overline{V})&=\frac{m+n}2\hbox{Trace}[(I_m-ZZ^*)^{-1}U(I_n-Z^*Z)^{-1}V^*],\\
 R_{II}&=\{Z\in\cal{M}_{n}: Z=Z^T, I_n-ZZ^*>0\},\hbox{ for }n\ge 1,\\
H_Z(U,\overline{V})&=\frac{n+1}2\hbox{Trace}[(I_n-ZZ^*)^{-1}U(I_n-Z^*Z)^{-1}V^*],\\
R_{III}&=\{Z\in\cal{M}_{n}: Z=-Z^T, I_n-ZZ^*>0\},\hbox{ for }n\ge 2,\\
H_Z(U,\overline{V})&=\frac{n-1}2\hbox{Trace}[(I_n-ZZ^*)^{-1}U(I_n-Z^*Z)^{-1}V^*],\\
R_{IV}&=\left\{z\in\BigC^n: \left|\sum z_j^2\right|^2+1-2\|z\|^2>0,\,\left|\sum z_j^2\right|^2<1\right\},\,1\le n\ne 2,\\
H_z(u,\overline{v})&=nAu[A(I_n-z^T\overline{z})+(I_n-z^T\overline{z})Z^*z(I_n-z^T\overline{z})]v^*,\end{aligned}$$
where $z^T$ is the transpose of $z$ and $A=|\sum_{j=1}^nz_j^2|^2+1-2\|z\|^2$.

Since some irreducible domains may belong to different classes, we add the following dimensional restrictions that yield uniqueness: $n\ge 2$ for domains in $R_{II}$, and $n\ge 5$ for domains in $R_{III}$ and $R_{IV}$.

\section{Bloch semi-norm as a Lipschitz number}

In this section, we give an alternate description of the Bloch semi-norm and analyze some useful properties of Bloch functions.

We begin by observing that the Bloch functions on a bounded homogeneous domain $D$ are precisely the Lipschitz maps between $D$ under the distance $\r$ induced by the Bergman metric and the complex plane under the Euclidean distance. To see this, we recall a result of \cite{G} connecting local derivatives to Lipschitz mappings.
If we consider $f:D\to\BigC$ as a map between Riemannian manifolds, then $Q_f(z)$ is exactly the operator norm of $df(z)$, the induced linear transformation on the tangent space at $z\in D$, with respect to the Bergman metric $H_z$ on $D$ and the Euclidean metric on $\BigC$.  Thus, $\b_f$ is the supremum of $\|df(z)\|$ over all $z\in D$.  The dilation ${\rm{dil}}(f)$  of $f$ is the global Lipschitz number of $f$: $${\rm{dil}}(f)=\sup_{\stackrel{\scriptstyle z\ne w}{z,w\in D}}\displaystyle\frac{|f(z)-f(w)|}{\r(z,w)}.$$

A {\bf length space} is a Riemannian manifold in which the distance
between two points is the infimum of the lengths of geodesics
connecting the points.  Examples include bounded homogeneous domains
under the Bergman metric. An example of a space which is not a
length space is the complement of a closed ball in $\BigC^n$ under
the Euclidean metric.  In Property 1.8 bis of \cite{G}, it is shown that
in the special case when $M$ is a length space and $f:M\to N$ is a
smooth map between Riemannian manifolds, $${\rm{dil}}(f) =
\sup_{x\in M}\|df(x)\|.$$ We deduce

\begin{theorem}\label{tlip} Let $D$ be a bounded homogeneous domain and let $f\in H(D,\BigC)$. Then $f$ is Bloch if and only if $f$ is a Lipschitz map as a function from $D$ under the Bergman metric and the complex plane under the Euclidean metric. Furthermore
$$\b_f=\sup_{z\ne w}\frac{|f(z)-f(w)|}{\r(z,w)}.$$
\end{theorem}

A proof that does not use differential geometry can be found in \cite{C2} (Theorem~10) and in \cite{Z} (Theorem~5.1.6) for the case of the unit disk, and in \cite{Z2} (Theorem~3.6) for the case of the unit ball in $\BigC^n$.

The following result is an extension of Theorem~2 of \cite{X} to any bounded homogeneous domain.

\begin{theorem}\label{bddcomp} Let $D$ be a bounded homogeneous domain and let $\f\in H(D)$. Then $C_\f$ is a bounded operator on the Bloch space of $D$. Furthermore
$$1\le \|C_\f\|\le \max\{1,\r(\f(0),0)+B_\f\}.$$ In particular, if $\f(0)=0$, then $1\le \|C_\f\|\le \max\{1,B_\f\}$.\end{theorem}

\begin{proof} By Theorem~\ref{tlip}, for any $f\in \cal{B}$, $|f(\f(0))-f(0)|\le \r(\f(0),0)\b_f$, whence
$$|f(\f(0))|\le |f(0)|+|f(\f(0))-f(0)|\le \|f\|_{\cal{B}}+(\r(\f(0),0)-1)\b_f.$$
Thus, using (\ref{semibdd}) we obtain
$$\|f\circ \f\|_{\cal{B}}=|f(\f(0))|+\b_{f\circ\f}\le \|f\|_{\cal{B}}+(\r(\f(0),0)-1+B_\f)\b_f.$$
If $\r(\f(0),0)-1+B_\f\le 0$, then $\|f\circ \f\|_{\cal{B}}\le \|f\|_{\cal{B}}.$ If
$\r(\f(0),0)-1+B_\f\ge 0$, then $$\|f\circ \f\|_{\cal{B}}\le (\r(\f(0),0)+B_\f)\|f\|_{\cal{B}}.$$
Hence $C_\f$ is bounded on $\cal{B}$ and $$\|C_\f\|=\sup_{\|f\|_{\cal{B}}=1}\|f\circ \f\|_{\cal{B}}\le \max\{1,\r(\f(0),0)+B_\f\}.$$ The lower estimate can be deduced immediately by taking as a test function $f$ the constant function 1. \end{proof}

The following result is the extension to the unit ball of Corollary~1 and Theorem~2 in \cite{X}.

\begin{corollary}\label{lowest} For any $\f\in H(\BigB_n)$,
$$\max\left\{1,\frac12\log\frac{1+\|\f(0)\|}{1-\|\f(0)\|}\right\}\le \|C_\f\|\le \max\left\{1,\frac12\log\frac{1+\|\f(0)\|}{1-\|\f(0)\|}+B_\f\right\}.$$ In particular, if $\f(0)=0$, then $\|C_\f\|=1$.
\end{corollary}

For the proof we need the following result which is part of Theorem~3.1 of \cite{Z2}.

\begin{lemma}\label{charqf} {\rm \cite{Z2}} For all $f\in H(\BigB_n,\BigC)$, $z\in \BigB_n$,
\begin{equation} Q_f(z)=[(1-\|z\|^2)(\|(\nabla f)(z)\|^2-|Rf(z)|^2)]^{1/2},\label{charq}\end{equation}
where $Rf(z)=\de\sum_{k=1}^n z_k\frac{\partial f}{\partial z_k}(z)$. Therefore
$$\b_f=\sup_{z\in\BigB_n} [(1-\|z\|^2)(\|(\nabla f)(z)\|^2-|Rf(z)|^2)]^{1/2}.$$
\end{lemma}

\begin{proof} The upper estimate follows immediately from Theorem~\ref{bddcomp}, since by  Proposition~1.21 of \cite{Z2}, $$\r(\f(0),0)=\frac12\log\frac{1+\|\f(0)\|}{1-\|\f(0)\|}.$$
  To prove the lower estimate, note that if $\f(0)=0$, then the result follows from Theorem~\ref{bddcomp}. So assume $\f(0)\ne 0$. For $z\in\BigB_n$, define
$$f(z)=\frac12\log\frac{\|\f(0)\|+\langle z,\f(0)\rangle}{\|\f(0)\|-\langle z,\f(0)\rangle}.$$ Then $f$ is holomorphic on $\BigB_n$, $f(0)=0$, and $f(\f(0))=\frac12\log\frac{1+\|\f(0)\|}{1-\|\f(0)\|}$.
Moreover, using the triangle inequality, we obtain
$$\begin{aligned}\|(\nabla f)(z)\|^2-|Rf(z)|^2&=\frac{\|\f(0)\|^2(\|\f(0)\|^2-|\langle z,\f(0)\rangle|^2)}{|\|\f(0)\|^2-\langle z,\f(0)\rangle^2|^2}\\
&\le \frac{\|\f(0)\|^2}{|\|\f(0)\|^2-\langle z,\f(0)\rangle^2|}.\end{aligned}$$
Thus, by (\ref{charq}), we deduce
$$Q_f(z)^2\le \frac{(1-\|z\|^2)\|\f(0)\|^2}{|\|\f(0)\|^2-\langle z,\f(0)\rangle^2|}=\frac{1-\|z\|^2}{|1-\langle z,u_0\rangle^2|}\le 1,$$
where $u_0=\frac{\f(0)}{\|\f(0)\|}$. On the other hand, it is immediate to check that $Q_f(0)=1$. Therefore, $f\in \cal{B}$ and $\b_f=1=\|f\|_{\cal{B}}$. Hence $\|C_\f\|\ge \|f\circ \f\|_{\cal{B}}\ge |f(\f(0))|$. Since $\|C_\f\|\ge 1$, the lower estimate follows at once. In Section~5 (see paragraph after Theorem~\ref{koranyi}), it will be shown that $B_\f\le 1$. Thus, in the case when $\f(0)=0$, we deduce that $\|C_\f\|=1$. \end{proof}

The characterization of the Bloch semi-norm in Theorem~\ref{tlip} leads to the following convergence theorem.

\begin{theorem}\label{semicont} Let $\{f_n\}$ be a sequence of Bloch functions on a bounded homogeneous domain $D$ which converges locally uniformly in $D$ to some holomorphic function $f$.  If the sequence $\{\b_{f_n}\}$ is bounded, then $f$ is Bloch and
$$\b_f \leq \liminf_{n \to \infty} \b_{f_n}.$$  That is, the
function $f \mapsto \b_f$ is lower semi-continuous on
$\cal{B}$ under the topology of uniform convergence on compact subsets of $D$.\end{theorem}

\begin{proof}  Since $\{\b_{f_n}\}$ is bounded, then $C =
\liminf_{n \to \infty} \b_{f_n}$ exists and is non-negative. Also,
there exists a subsequence $\{\b_{f_{n_k}}\}$ which converges to
$C$.  If $C = 0$, then $f$ is constant and the result follows at once.
So, we assume $C > 0$.

Let $z,w\in D$ and fix $\varepsilon>0$. Choose a positive integer $\n$ so that $|f_{n_k}(z)-f(z)|<\varepsilon/2$, $|f_{n_k}(w)-f(w)|<\varepsilon/2$, and $\b_{f_{n_k}}<C+\varepsilon$, for all $k\ge \n$. Then
$$|f(z) - f(w)| < \varepsilon + \b_{f_{n_k}}\r(z,w) < \varepsilon(1+\r(z,w)) +
C\r(z,w).$$
Letting $\varepsilon\to 0$, we obtain $\left|f(z) - f(w)\right| \leq C\r(z,w)$. Using Theorem~\ref{tlip}, we deduce that $f$ is Bloch and $\b_f \leq C$. \end{proof}

\section{Bloch constant of a bounded symmetric domain}

For a bounded symmetric domain $D$ define the {\bf Bloch constant of $D$} as
$$c_D=\sup\{\b_f: f\in H(D,\BigD)\}$$ and the {\bf inner radius of $D$} as
$$r_D=\inf_{u\in \partial D'}H_0(u,\overline{u})^{1/2},$$ where $D'$ is a bounded symmetric domain in standard form biholomorphically equivalent to $D$. By Theorem~2 of \cite{CC}, if $D$ is a Cartan classical domain, then
\begin{equation} c_D=\frac1{r_D}=\begin{cases} \sqrt{2/(m+n)}&\hbox{ if }D\in R_I,\\
\sqrt{2/(n+1)}&\hbox{ if }D\in R_{II},\\
\sqrt{1/(n-1)}&\hbox{ if }D\in R_{III},\label{maine}\\
\sqrt{2/n}&\hbox{ if }D\in R_{IV}.\end{cases}\end{equation}
In particular, if $D$ is the unit polydisk, then $c_D=1$. If $D=\B^n$, then
$c_D=\sqrt{2/(n+1)}$.

In \cite{ZG}, formula (\ref{maine}) was extended to include the case of the exceptional domains:  $$c_D=\frac1{r_D}=\begin{cases}1/\sqrt{6}&\hbox{ if }D= R_V,\\
1/3&\hbox{ if }D= R_{VI}.\nonumber\end{cases}$$

Furthermore, by Theorem~3 of \cite{CC} extended to include the exceptional domains, if $D=D_1\times \cdots\times D_k$ is in standard form, then $c_D=\max_{1\le j\le k}c_{D_j}$.

In the following proposition, for each of the four types of Cartan classical domains, we give explicit functions with range contained in $\BigD$ which are extremal in the sense that their semi-norms are equal to the Bergman constants of the domains. The proof of the following lemma can be found in \cite{CC} (Lemma~2 and Lemma 3).

\begin{lemma}\label{zless}{\normalfont{\cite{CC}}} {\normalfont{(a)}} Let $Z$ be a matrix such that $I-ZZ^*$ is positive definite. Then each row of $Z$ has norm less than 1.

{\normalfont{(b)}} If $D\in R_{IV}$, then for each $z=(z_1,\dots,z_n)\in D$ and each distinct $r,s\in \{1,\dots,n\}$, $z_r\pm iz_s\in \BigD$.
\end{lemma}

\begin{proposition}\label{proj} {\normalfont{(a)}} If $D$ is a Cartan classical domain in $R_I$, $R_{II}$, or $R_{III}$, then the canonical projections $p_j(z)=z_j$ $(\hbox{for }j=1,\dots,{\rm{dim}}(D))$ map $D$ into  $\BigD$ and have Bloch semi-norm equal to $c_D$.

{\normalfont{(b)}} If $D\in R_{IV}$, then for each pair of distinct indices $r,s\in\{1,\dots,n \}$, for $j=1,2$, the functions $p^{j}_{r,s}$ defined on $D$ by $p^{1}_{r,s}(z)=z_r+iz_s$ and  $p^{2}_{r,s}(z)=z_r-iz_s$ have range contained in $\BigD$ and semi-norm equal to $c_D$.
\end{proposition}

For easier reference, we shall refer to the functions in (b) as {\bf modified projection maps}.

\begin{proof} (a) The projections $p_j$ map $D$ into $\BigD$ by part (a) of Lemma~\ref{zless}. Represent the elements of $D\in R_I\cup R_{II}\cup R_{III}$ as vectors by considering the free variables in their matrix representation arranged rowwise. If $Z\in D$ with $D\in R_I$, then each entry of $Z$ is a free variable and thus $Z=(z_{1,1},z_{1,2},\dots,z_{1,n}z_{2,1},z_{2,2},\dots,z_{m,n})$. If $j\in \{1,\dots,mn\}$, let $q$ and $r$ be the unique integers such that $j=nq+r$ with $0\le q<m$ and $0\le r<n$ or $q=m$ and $r=0$. Define $e_j$ to be the unit vector whose $j$th coordinate is 1 and all others are 0, which corresponds to the matrix $E_{q,r}$ whose $(q,r)$-entry is 1 and all other entries are 0. Then $$H_0(e_j,\overline{e_j})=\left(\frac{m+n}2\right)\hbox{Trace}(E_{q,r}E_{q,r}^*)=\frac{m+n}2.$$ Furthermore, $|\nabla p_j(0)e_j|=1$ so that $Q_{p_j}(0)=\sqrt{\frac2{m+n}}=c_D$.

If $Z\in D$ with $D\in R_{II}$, then the free variables are the entries on and above the main diagonal. Thus, for $j\in\{1,\dots,n(n+1)/2\}$ and the appropriate choice of indices $h$ and $k$ corresponding to $j$, we have $$H_0(e_j,\overline{e_j})=\left(\frac{n+1}2\right)\hbox{Trace}(E_{h,k}E_{h,k}^*)=\frac{n+1}2.$$ Furthermore, $|\nabla p_j(0)e_j|=1$ so that $Q_{p_j}(0)=\sqrt{\frac2{n+1}}=c_D$.

If $Z\in D$ with $D\in R_{III}$, then the free variables are the entries above the main diagonal. For $j\in\{1,\dots,n(n-1)/2\}$ and the appropriate choice of indices $h$ and $k$ corresponding to $j$, let $X_{h,k}=E_{h,k}-E_{k,h}\in \partial D$. Then  $$H_0(e_j,\overline{e_j})=\left(\frac{n-1}2\right)\hbox{Trace}(X_{h,k}X_{h,k}^*)=n-1.$$ Furthermore, $|\nabla p_j(0)e_j|=1$ so that $Q_{p_j}(0)=\frac1{\sqrt{n-1}}=c_D$.

(b) Let $D\in R_{IV}$. The range of the modified projection maps $p^1_{r,s}$ and $p^2_{r,s}$ is contained in $\BigD$ by part (b) of Lemma~\ref{zless}. For $j=1,2$, let $e^j_{r,s}$ be the vector in $\BigC^n$ that has $1/2$ in the $r$-th place, $(-1)^{j} i/2$ in the $s$-th place, and 0 elsewhere. Thus $e^j_{r,s}\in \partial D$ and $H_0(e^j_{r,s},\overline{e}^j_{r,s})=n/2$. Moreover, $\nabla p^j_{r,s}(0)$ is the vector that has 1 in the $r$-th place, $(-1)^{j+1} i$ in the $s$-th place, and 0 elsewhere, so $|\nabla p^j_{r,s}(0)e^j_{r,s}|=1$ and thus $Q_{p^j_{r,s}}(0)=\sqrt{\frac2{n}}=c_D$, completing the proof.
\end{proof}

\begin{corollary}\label{corpo} If $D=D_1\times\dots\times D_k$ is a bounded symmetric domain, with $D_1,\dots,D_k$ Cartan classical domains, then the projection maps or the modified projection maps for the case of a factor in $R_{IV}$ of $D$ onto $\BigD$ have semi-norms $c_{D_1},\dots,c_{D_k}$ in this order, repeated as many times as the dimensions of the factors $D_1,\dots,D_k$.
\end{corollary}

\begin{proof} Since the Bergman metric of $D_1\times\dots\times D_k$ is the direct sum of the Bergman metrics of the individual factors, the result follows at once from Proposition~\ref{proj}. \end{proof}

In Section 6, we shall prove that if $D=D_1\times\dots\times D_k$, where each $D_j$ is a Cartan classical domain, and $C_\f$ is an isometry on $\cal{B}$, then the components $\f_j=p_j\circ \f$ or the functions $\f_r+i\f_s$ and $\f_r-i\f_s$ (which we call {\bf modified components}) must have Bloch semi-norms equal to $c_{D_1},\dots,c_{D_k}$, respectively, repeated according to the dimension of the factors.

\section{The holomorphic self-maps on a bounded homogeneous domain}

A bounded homogeneous domain $D$ considered as a Riemannian manifold under the structure induced by the Bergman metric is a length space. Therefore, given $\f\in H(D)$ viewed as a map between Riemannian manifolds, we may apply Property 1.8 bis of \cite{G}. We obtain the following interpretation of the Bergman constant.

\begin{proposition}\label{psame} Any function $\f\in H(D)$ is Lipschitz as a mapping from $(D,\r)$ into itself and the Bergman constant $B_\f$ is the Lipschitz number of $\f$. That is
$$B_\f=\sup_{z\ne w}\frac{\r(\f(z),\f(w))}{\r(z,w)}.$$
\end{proposition}

Using Proposition~\ref{psame} and the invariance of the Bergman distance under the action of an automorphism of $D$, we obtain

\begin{corollary}\label{pres} Let $D$ be a bounded homogeneous domain and let $\f\in H(D)$. Then for all $S\in{\rm{Aut}}(D)$, $B_{\f\circ S}=B_\f=B_{S\circ \f}$.
\end{corollary}

The proof of the following result is analogous to that of Theorem~\ref{semicont} where the Euclidean distance is replaced by the Bergman distance.

\begin{corollary}\label{semcont} Let $D$ be a bounded homogeneous domain. If $\{\f_n\}$ is a sequence in $H(D)$ converging to a function $\f\in H(D)$ uniformly on compact subsets of $D$, then $B_\f\le \liminf_{n\to\infty}B_{\f_n}$.
\end{corollary}

\begin{theorem}\label{lsuf} Let $D$ be a bounded homogeneous domain and let $\f\in H(D)$ be such that $\f(0)=0$ and suppose there exists a sequence $\{S_j\}$ in $\hbox{Aut}(D)$ such that $\{\f\circ S_j\}$ converges locally uniformly in $D$ to the identity. If the Bergman constant of $\f$ does not exceed $1$, then $C_\f$ is an isometry on $\cal{B}$.
\end{theorem}

\begin{proof} First observe that, by Corollary~\ref{semcont} and Corollary~\ref{pres}, $1\le B_{\f\circ S_j}=B_\f$. Thus, from the hypothesis it follows that $B_\f=1$. If $f\in\cal{B}$, then $f\circ\f\circ S_n\to f$ locally uniformly in $D$, so by Theorem~\ref{semicont} and the invariance of the Bloch semi-norm under composition of automorphisms, we obtain
$$\b_f\le \liminf_{n\to\infty} \b_{f\circ \f\circ S_n}=\b_{f\circ \f}\le B_\f\b_f=\b_f.$$
Thus, since $\f(0)=0$, $\|f\circ \f\|_{\cal{B}}=|f(0)|+\b_{f\circ \f}=|f(0)|+\b_f=\|f\|_{\cal{B}}$.
\end{proof}

In \cite{Ko}, Kor\'anyi proved the following version of the Schwarz lemma for bounded symmetric domains.

\begin{theorem} \label{koranyi} {\normalfont{\cite{Ko}}} Let $D$ be a bounded symmetric domain of rank $\ell$ and let $\f\in H(D)$. Then for all $z,w\in D$, $\r(\f(z),\f(w))\le \ell^{1/2}\r(z,w),$
where $\r$ is the Bergman distance on $D$. The constant $\ell^{1/2}$ is the best possible.
\end{theorem}

The rank of $D$ is roughly the dimension of the largest polydisk that can be embedded in the tangent space at each point of $D$, considered as a complex manifold under the Bergman metric. In particular, if $D=\BigD^n$, then $\ell=n$ so that $\max_{\f\in H(D)}B_\f= \sqrt{n}$. If $D$ is the unit ball, $\ell=1$ so that $\max_{\f\in H(D)}B_\f= 1$.  Thus, in the statement of Theorem~\ref{lsuf} for the case $D=\BigB_n$, the hypothesis on $B_\f$ is unnecessary.

Theorem~\ref{lsuf} allows us to obtain non-trivial examples of isometric composition operators if $D$ is a bounded homogeneous domain having $\BigD$ as a factor.

\begin{example} Let $D=D_1\times \BigD$, where
$D_1$ is a bounded homogeneous domain containing 0. Let $\phi\in
H(\BigD)$ be such that $\phi(0)=0$, $\b_\phi=1$, and suppose $\phi$
is not a rotation. Let $U\in {\rm{Aut}}(D_1)$ such that $U(0)=0$.
For $z\in D_1$ and $\z\in\BigD$, define
$\f(z,\zeta)=(U(z),\phi(\zeta))\in D$. By the equivalence of (a) and
(g) in Theorem~\ref{cmain}, there exists a sequence $\{T_j\}$ in
$\hbox{Aut}(\BigD)$ such that $\phi\circ T_j(\z)\to \z$ locally
uniformly in $\BigD$. Then the sequence $\{S_j\}$ defined by
$S_j(z,\zeta)=(U^{-1}(z),T_j(\zeta))$ is in $\hbox{Aut}(D)$ and
$\{\f\circ S_j\}$ converges locally uniformly to the identity in
$D$.  Since the Bergman constants of $U$ and $\phi$ are both 1, then
so is the Bergman constant of $\f$.  By Theorem~\ref{lsuf}, $C_\f$
is an isometry on $\cal{B}$.\end{example}

\section{Some necessary conditions}

In this section, we derive some necessary conditions for $\f\in H(D)$ to be the symbol of an isometric composition operator on the Bloch space when $D$ is a bounded symmetric domain in standard form with no exceptional factors. These conditions are analogous to those found in \cite{CC3} in the special case of the unit polydisk for domains that do not have factors in $R_{IV}$. A modification is needed when dealing with domains that contain factors in $R_{IV}$. The key tool in the proof of  Theorem~\ref{neccond} below is the existence of an automorphism interchanging any prescribed point of $D$ with 0.

\begin{lemma}\label{exinv} Let $D$ be a bounded symmetric domain in standard form. For each $a\in D$, there exists and involution $\psi_a\in \hbox{Aut}(D)$ such that $\psi_a(a) = 0$.
\end{lemma}

The proof is an immediate consequence of the results in \cite{He}, pp. 170, 301, 311.

\begin{theorem}\label{neccond} Let $D=D_1\times\dots\times D_k$ be a bounded symmetric domain in standard form. Let $\f=(\f_1,\dots,\f_n)\in H(D)$ such that $C_\f$ is an isometry on $\cal{B}$. Then:
\begin{enumerate}
\item[{\normalfont{(a)}}] The components $\f_1,\dots,\f_n$ are linearly independent.

\item[{\normalfont{(b)}}] If $D$ does not contain any exceptional factors, then $\f(0)=0$.

\item[{\normalfont{(c)}}] If none of the factors of $D$ is in $R_{IV}$, then the components of $\f$ have semi-norms equal to $c_{D_1},\dots,c_{D_k}$, repeated according to the dimension of each factor.

\item[{\normalfont{(d)}}] If $D$ has factors in $R_{IV}$, then for each $D_\ell\in R_{IV}$ and each pair $r,s$ of distinct indices, with
$\sum_{i=1}^{\ell-1}\hbox{dim}(D_i)<r,s\le \sum_{i=1}^{\ell}\hbox{dim}(D_i)$, the modified components $\f_r+i\f_s,\f_r-i\f_s\in H(D,\BigD)$ and have semi-norm equal to $c_{D_\ell}$.
\end{enumerate}
\end{theorem}

\begin{proof} To prove (a), assume there exist $k=1,\dots,n$ and $\a_j\in\BigC$ for $j\ne k$ such that $\f_k=\sum_{j\ne k}\a_j \f_j$. Then the function $f$ defined by $f(z)=z_k-\sum_{j\ne k}\a_jz_j$ is bounded on $D$ (and hence a Bloch function) and has nonzero Bloch norm, but $f\circ \f$ is identically zero. Thus $C_\f$ is not an isometry on $\cal{B}$.

To prove (b), assume each of the factors of $D$ is a Cartan
classical domain. Set $\f(0)=a$ and write $a=(a_1,\dots, a_n)$, with
each $a_j\in\C$. Fix $j=1,\dots,n$. By Lemma~\ref{exinv}, we may
consider $\psi_j(z)=p_j\circ \psi_a$, where $p_j$ is the projection
map onto the $j$-th coordinate and $\psi_a$ is an automorphism of
$D$ mapping $a$ to 0 and 0 to $a$. Then, by the invariance of the
semi-norm under composition of automorphisms of $D$ and by
Corollary~\ref{corpo}, if $D_\ell\notin R_{IV}$ and
$\hbox{dim}(D_1\times\dots\times
D_{\ell-1})<j\le\hbox{dim}(D_1\times\dots\times D_{\ell})$, then
$\b_{\psi_j}=\b_{p_j}=c_{D_\ell}$. On the other hand,
$\b_{\psi_j\circ \f}\le c_{D_\ell}$.  Since $\|\psi_j\circ
\f\|_{\cal{B}}=\|\psi_j\|_{\cal{B}}$, and $\psi_j(a)=0$, we deduce
$|a_j|+c_{D_\ell}=\b_{\psi_j\circ\f}\le c_{D_\ell}$. Hence $a_j=0$
and $\b_{\psi_j\circ\f}= c_{D_\ell}$.

 Next assume $D_\ell\in R_{IV}$ and $\hbox{dim}(D_1\times\dots\times D_{\ell-1})<r,s\le\hbox{dim}(D_1\times\dots\times D_{\ell})$, with $r\ne s$. Define $\psi^h_{r,s}=p^h_{r,s}\circ \psi_a$, where $p^h_{r,s}$ ($h=1,2$) are the modified projection maps defined in Proposition~\ref{proj}. Then $\b_{\psi^h_{r,s}\circ \f}\le c_{D_\ell}$, and since $\|\psi^h_{r,s}\circ \f\|_{\cal{B}}=\|\psi^h_{r,s}\|_{\cal{B}}$, and $\psi^h_{r,s}(a)=0$, by Corollary~\ref{corpo} we deduce
$$|a_r+(-1)^{h+1} ia_s|+c_{D_\ell}=\b_{\psi^h_{r,s}\circ\f}\le c_{D_\ell}.$$ Hence $a_r+(-1)^{h+1} ia_s=0$ for $h=1,2$, which implies that $a_r=a_s=0$. Consequently,  $\f(0)=0$.

Moreover, $\b_{\f_j}=\b_{p_j\circ \f}= \b_{p_j}=c_{D_\ell}$ if $D_\ell\notin R_{IV}$, while
$\b_{\f_r+(-1)^{h+1} i\f_s}=\b_{p^h_{r,s}\circ \f}= \b_{p^h_{r,s}}=c_{D_\ell}$ if $D_\ell\in R_{IV}$, proving (c) and (d).
\end{proof}

\begin{remark} We suspect that if $D$ is any bounded symmetric domain in standard form, then the condition $\f(0)=0$ holds, but we have not been able to verify it in the case an irreducible factor is an exceptional domain because our proof relies on our ability to find explicit functions mapping into the unit disk that are extremal with respect to the Bloch semi-norm and attain the semi-norm at 0 (these were the projection maps for domains in $R_I,R_{II},$ and $R_{III}$, and the modified projection maps for domains in $R_{IV}$). We have not been able to determine any functions defined on the exceptional domains that have these properties.
\end{remark}

Thus, we propose the following

\begin{conjecture}Let $\f\in H(D)$, where $D=D_1\times\dots \times D_k$ is a bounded symmetric domain in standard form. Then $C_\f$ is an isometry on $\cal{B}$ if and only if $\f(0)=0$, $B_\f=1$, and the components or the modified components of $\f$ have semi-norm equal to $c_{D_1},\dots,c_{D_k}$, repeated according to the dimension of each factor.\end{conjecture}


\section{The Spectrum of $C_\f$ in the Case of the Polydisk}

In this section, we study the spectrum of a large class of
isometric composition operators on the Bloch space of the unit
polydisk $\BigD^n$.  As a consequence, we obtain the
classification of the spectrum in the case of the unit disk found in \cite[Theorem 5.1]{AC}.

First, we recall pertinent facts of the spectrum of bounded linear
operators on a complex Banach space.  These results can all be found
in a standard text of functional analysis such as \cite{CON}.

A bounded linear operator $A$ on a complex Banach space $E$ is
invertible if and only if it has dense range and is \textit{bounded
below}, that is, there exists $c > 0$ such that for all $x \in E$,
$\norm{Ax}\geq c\norm{x}$. By the Inverse Mapping Theorem, for
invertibility it is sufficient for $A$ to be bijective.

The \textbf{resolvent} of a bounded linear operator $A$ on a complex
Banach space $E$ is defined as
$$\rho(A) = \{\lambda \in \C : A-\lambda I \text{ is invertible}\},$$ where $I$ is the identity operator.
The \textbf{spectrum} of $A$ is defined as $\sigma(A) =
\C\setminus\rho(A)$.  The \textbf{approximate point spectrum}, a
subset of the spectrum, is defined as
$$\sigma_{ap}(A) = \{\lambda \in \C : A-\lambda I \text{ is not
bounded below}\}.$$

The spectrum of an operator $A$ is a non-empty, compact subset of
$\C$ contained in the closed disk centered at the origin of radius
equal to the operator norm of $A$.
In particular, the
spectrum of an isometry is contained in $\overline{\BigD}$, since
the operator norm is 1.  Furthermore
$\partial\sigma(A) \subseteq \sigma_{ap}(A)$ (see \cite{CON}, Proposition~6.7).

To prove our results on the spectrum, we need the following
lemma.
\begin{lemma}\label{conway_lemma} Let $E$ be a complex Banach space and suppose
$T:E \to E$ is
an isometry.  If $T$ is invertible, then $\sigma(T) \subseteq
\partial\BigD$.  If $T$ is not invertible, then $\sigma(T) =
\overline{\BigD}$.\end{lemma}

\begin{proof} Suppose $T$ is an invertible isometry on $E$.  Then $0
\not\in \sigma(T)$, and so the function $z \mapsto z^{-1}$ is analytic
in some neighborhood of $\sigma(T)$.  By the Spectral Mapping
Theorem (see \cite{CON}, p. 204), we have $\sigma(f\circ T) = f(\sigma(T))$, and so
$$\sigma(T^{-1}) = \sigma(T)^{-1} = \{\lambda^{-1} : \lambda \in
\sigma(T)\}.$$  Since $T^{-1}$ exists and is an isometry, we have
$\sigma(T^{-1}) \subseteq \overline{\BigD}$.  Therefore $\sigma(T)
\subseteq \partial\BigD$.

Next, suppose $T$ is not invertible.  In order to prove that
$\sigma(T) = \overline{\BigD}$, it suffices to show that
$\overline{\BigD} \subseteq \sigma(T)$.  For $\lambda \in \BigD$,
$T-\lambda I$ is bounded below by $1-\modu{\lambda}$.  Thus,
$\lambda \not\in \sigma_{ap}(T)$.  We deduce that
$\partial\sigma(T)\subseteq \sigma_{ap}(T) \subseteq \partial\BigD$.

Since $T$ is not invertible, $0 \in \sigma(T)$.  Assume $\lambda \in
\overline{\BigD}\cap\rho(T)$.  Note that $\lambda \not\in
\partial\sigma(T)$ since $\partial\sigma(T) = \sigma(T) \cap
\overline{\rho(T)}$.  Consider $\Gamma = \{t\lambda : t \in
[0,\infty)\}$, the radial line through $\lambda$.  Since $\sigma(T)$
is closed, there exists $t \in [0,1)$ such that $t\lambda \in
\partial\sigma(T)$.  This contradicts the fact that
$\partial\sigma(T) \subseteq \partial\BigD$.  Consequently,
$\overline{\BigD}\cap\rho(T) = \emptyset$, whence
$\overline{\BigD} \subseteq \sigma(T)$. \end{proof}

We use Lemma~\ref{conway_lemma} to obtain a classification of the spectrum for a
large class of symbols of isometric composition operators for the polydisk.
We first describe the spectrum of isometric composition operators
with surjective, non-automorphic symbol.

\begin{proposition}\label{non-auto_onto} If $\varphi$ is the symbol of an isometric
composition operator such that $\varphi \not\in \Aut(\BigD^n)$ and
$\varphi$ is onto, then $\sigma(C_\varphi) =
\overline{\BigD}$.\end{proposition}

\begin{proof} By Lemma \ref{conway_lemma}, it suffices to show that $0 \in
\sigma(C_\varphi)$. Since $C_\f$ is an isometry, it is necessarily
injective. Thus, we must prove that $C_\varphi$ is not surjective.
Arguing by contradiction, suppose $C_\varphi$ is surjective. Since
$\varphi$ is onto but not an automorphism, $\varphi$ is not 1-1.
Thus there exist distinct $\zeta,\eta \in \BigD^n$ such that $\varphi(\zeta) = \varphi(\eta)$.  In particular, there
exists $k=1,\dots,n$ such that $\zeta_k \neq \eta_k$.

  Let $p_k:\BigD^n \to \BigD$ be the
projection map onto the $k^{th}$ coordinate, and define the Bloch
function $g(z) = p_k(z) - \eta_k$.  Note that $g$ is non-vanishing
at $\zeta$.  Since $C_\varphi$ is surjective, there exists a Bloch
function $f$ such that $f\circ\varphi = g$.  In particular
$f(\varphi(\zeta)) = g(\zeta) \neq 0.$  On the other hand,
$f(\varphi(\zeta)) = f(\varphi(\eta)) = g(\eta) = 0,$ a
contradiction. Thus,
 $C_\varphi$ is not surjective.  Therefore $\sigma(C_\varphi) =
\overline{\BigD}$. \end{proof}

Observe that the class of symbols described in Proposition~\ref{non-auto_onto} includes those of the form $\varphi =(\varphi_1, \dots,\varphi_n)$ with each $\varphi_k$ dependent on a distinct single complex variable as described in
Theorem \ref{cmain}(f).

 Next, we consider the isometric
composition operators induced by a particular type of automorphic
symbol. Let $\varphi(z) = (\lambda_1z_1,\dots,\lambda_nz_n)$, with
$\modu{\lambda_j} = 1$ for $j = 1,\dots, n$.  We define
$\mathrm{ord}(\lambda_j)$ to be the smallest non-negative integer $k$ such that
$\lambda_j^k = 1$.  We have the following classification of the
spectrum of $C_\varphi$ for this class of automorphisms.

\begin{proposition}\label{polyspec}
Let $\varphi(z) = (\lambda_1z_1,\dots,\lambda_nz_n)$ such that
$\modu{\lambda_j} = 1$ for $j=1,\dots,n$.  Then
$$\sigma(C_\varphi) =
\begin{cases}
\partial\BigD,& \text{if } \mathrm{ord}(\lambda_j) = \infty \text{ for
some } $j$,\\
G,& \text{if } \mathrm{ord}(\lambda_j)<\infty
\text{ for all } $j$,
\end{cases}$$
where $G$ is the finite cyclic group generated
by $\lambda_1,\dots,\lambda_n$.
\end{proposition}

\begin{proof}
The induced composition operator $C_\varphi$ is clearly invertible
since $\varphi^{-1}(z) =
(\lambda_1^{-1}z_1,\dots,\lambda_n^{-1}z_n)$ and $C_\varphi^{-1} =
C_{\varphi^{-1}}$.  So, by Lemma \ref{conway_lemma}, we have
$\sigma(C_\varphi) \subseteq \partial\BigD$.

Observe that the group $G$ generated by $\l_1,\dots,\l_n$ is $$\{\lambda_1^{k_1}\cdots\lambda_n^{k_n} : k_1,\dots,k_n
\in\mathbb{N}\cup\{0\}\} \subseteq \partial\BigD.$$ For
$k_1,\dots,k_n \in \mathbb{N}\cup\{0\}$, the function $f(z) =
z_1^{k_1}\dots z_n^{k_n}$ is Bloch and
$(f\circ\varphi)(z) = f(\lambda_1z_1,\dots,\lambda_nz_n) =
\lambda_1^{k_1}\cdots\lambda_n^{k_n}f(z)$.  Thus, $f$ is an
eigenfunction of $C_\varphi$ with eigenvalue
$\lambda_1^{k_1}\cdots\lambda_n^{k_n}$.  Thus,
$\overline{G} \subseteq \sigma(C_\varphi)$.

If there exists $j=1,\dots, n$ such that $\mathrm{ord}(\lambda_j) =
\infty$, then $G$ is a dense subset of $\partial\BigD$.  Thus,
$\partial\BigD = \overline{G} \subseteq \sigma(C_\varphi).$ and
hence $\sigma(C_\varphi) = \partial\BigD$.  If
$\mathrm{ord}(\lambda_j) < \infty$ for all $j=1,\dots,n$, then $G$
is a finite cyclic group of order equal to the least common multiple
of the orders of $\lambda_1,\dots,\lambda_n$. We wish to show that
$\sigma(C_\varphi) \subseteq G$.  Let $\mu \in
\partial\BigD\backslash G$. 
The invertibility of $C_\varphi -
\mu I$ is equivalent to showing that for all Bloch functions $g$,
there exists a unique Bloch function $f$ such that $C_\varphi(f) -
\mu f = g$.

If $\alpha$ is the order of $G$, then $\varphi^\alpha(z) :=
(\underbrace{\varphi\circ\cdots\circ\varphi}_{\text{$\alpha$-times}})(z)
= z$ for all $z \in \BigD^n$, so $C_\varphi - \mu I$ is invertible
if and only if for every Bloch function $g$ the following system has
a unique solution $f$ in the Bloch space:
\begin{equation}\label{spec-sys}
\begin{array}{lclcl}
f(\varphi(z)) & - & \mu f(z) & = & g(z) \\
f(\varphi^2(z)) & - & \mu f(\varphi(z)) & = & g(\varphi(z)) \\
& \vdots &  & $\vdots$ &  \\
f(z) & - & \mu f(\varphi^{\alpha-1}(z)) & = &
g(\varphi^{\alpha-1}(z)).
\end{array}
\end{equation}

Equivalently, (\ref{spec-sys}) can be posed as the matrix equation
$Ax = b$ where
$$A = \left[\begin{matrix}
    -\mu & 1 & 0 & 0 & \cdots & 0\\
    0 & -\mu & 1 & 0 & \cdots & 0\\
    \vdots & 0 & \ddots & \ddots & & \vdots\\
    \vdots & & \ddots & \ddots & \ddots & \vdots\\
    0 & & & \ddots & \ddots & 1\\
    1 & 0 & \cdots & \cdots & 0 & -\mu
    \end{matrix}\right], x = \left[\begin{matrix}
    f(z)\\f(\varphi(z))\\\vdots\\\vdots\\f(\varphi^{\alpha-2}(z))\\f(\varphi^{\alpha-1}(z))
    \end{matrix}\right], b =
    \left[\begin{matrix}
    g(z)\\g(\varphi(z))\\\vdots\\\vdots\\g(\varphi^{\alpha-2}(z))\\g(\varphi^{\alpha-1}(z))
    \end{matrix}\right].
$$
Direct calculation shows that $\mathrm{det}(A) =
(-1)^\alpha(\mu^\alpha-1)\ne 0$, since $\mu\notin G$. Thus,
$C_\varphi - \mu I$ is invertible. For $\mu \notin G$, the unique
solution $f$ of (\ref{spec-sys}) is a (finite) linear combination of
the functions $g\circ\varphi^{j-1}$, for $j=1,\dots,\a$, each of
which is Bloch. Thus $f$ is Bloch. Therefore $\sigma(C_\varphi) =
G$. \end{proof}

We now study the spectrum in the case when $\f$ is an automorphism that permutes the coordinates.

\begin{proposition}\label{perm} Let $\f(z)=(z_{\t(1)},\dots,z_{\t(n)})$, where $\t\in S_n$ is decomposed into a product of disjoint cycles $c_1,\dots,c_\ell$ of order $\a_1,\dots,\a_\ell$, respectively.
\begin{enumerate}
\item[{\normalfont{(a)}}] If $\l$ is an $\a_j$-th root of unity for some $j=1,\dots,\ell$, then $\l$ is an eigenvalue of $C_\f$.
\item[{\normalfont{(b)}}] The spectrum of $C_\f$ is the group of the $\a$-th roots of unity, where $\a$ is the order of $\t$.
\end{enumerate}
\end{proposition}

\begin{proof} To prove (a), assume $\l$ is an $\a_j$-th root of unity for some $j=1,\dots,\ell$. To prove that $\l$ is an eigenvalue we need to show that there exists a nonzero $f\in \cal{B}$ such that
\begin{equation} f(\f(z))=\l f(z)\hbox{ for all }z\in \BigD^n.\label{func} \end{equation} We are going to show that, in fact, there is a linear function $f$ satisfying this condition. Let $f(z)=\sum_{j=1}^n x_jz_j$, where the coefficients $x_j$ are to be determined. Equation (\ref{func}) is equivalent to a matrix equation $Bx=0$, where $x$ is the column vector with coordinates $x_1,\dots,x_n$ and $B$ is a matrix with all diagonal entries $-\l$ and whose rows (and columns) contain an entry 1 and all other entries 0. The matrix $B$ can be permuted to yield the block diagonal matrix
$$A=\begin{bmatrix}A_1&O&\dots&O\\
O&A_2&\dots&O\\
\vdots& &\ddots&\vdots\\
O&O&\dots&A_n\end{bmatrix}$$
where each matrix $A_k$ has order $\a_k\times \a_k$ and contains all diagonal entries 1, and each row contains (besides the entry 1) one entry $-\l$ and all other entries 0. The matrix $A_k$ corresponds to the cycle $c_k$. It is easy to see that $\det(A_j)=1-\l^{\a_j}=0$, so that $\det(A)=\prod_{k=1}^\ell \det(A_k)=0$. Thus $B$ itself is singular, and so $Bx=0$ has nontrivial solutions.

To prove (b), assume $\m$ is not an $\a$-th root of unity. Arguing as in the proof of Proposition~\ref{polyspec}, given a function $g\in \cal{B}$, we form the system (\ref{spec-sys}), whose coefficient matrix has nonzero determinant. Thus, for any function $g\in\cal{B}$ there exists a unique $f\in \cal{B}$ such that $f\circ\f-\m f=g$.

On the other hand, if $\m$ is  an $\a$-th root of unity, then by row reduction it is easy to see that the solvability of the system (\ref{spec-sys}) reduces to the condition
\begin{eqnarray} g(\f^{\a-2}(z))&+&\m g(\f^{\a-3}(z))+\m^2g(\f^{\a-4}(z))+\dots \nonumber\\
\dots&+&\m^{\a-3}g(\f(z))+\m^{\a-2}g(z)+\m^{\a-1}g(\f^{\a-1}(z))\equiv 0.\label{ideo}\end{eqnarray}
By choosing $g$ to be a polynomial of sufficiently large degree, it is possible to find one for which identity (\ref{ideo}) fails. Thus, $\m\in\s(C_\f)$.
\end{proof}

As observed in Section~2, the automorphisms of the polydisk that fix
the origin have the representation $\varphi(z) =
(\lambda_1z_{\tau(1)},\dots,\lambda_nz_{\tau(n)})$ where
$\modu{\lambda_j} = 1$ for all $j \in\{1,\dots,n\}$, and $\tau \in
S_n$.  The arguments used in the proofs of
Proposition~\ref{polyspec} and Proposition~\ref{perm} carry over to
this general case, where $\a$ is taken to be the least common
multiple of the orders of $\l_1,\dots,\l_n$, and $\t$. We deduce the
following result:

\begin{theorem}\label{spec-poly} Let $\varphi$ be the symbol of an isometric
composition operator on the Bloch space for the polydisk.
\begin{enumerate}
\item[{\normalfont{(a)}}] If $\varphi \not\in \Aut(\BigD^n)$ and $\varphi$ is onto, then
$\sigma(C_\varphi)=\overline{\BigD}$.

\item[{\normalfont{(b)}}] If $\varphi \in \Aut(\BigD^n)$, let $\l_1,\dots,\l_n\in \partial \BigD$ and let $\t\in S_n$ be such that $\varphi(z) = (\lambda_1z_{\t(1)},\dots,\lambda_nz_{\t(n)})$.
\begin{enumerate}
\item[{\normalfont{\noindent(i)}}] If at least one of the $\lambda_j$ has infinite order, then $\sigma(C_\varphi) =
\partial\BigD$.

\item[{\normalfont{\noindent(ii)}}] If each $\l_j$ has finite order, then $\s(C_\f)$ is the cyclic group $\G$ generated by $\l_1,\dots,\l_n$ and by the $m^{\rm{th}}$ roots of unity, where $m$ is the order of $\t$. Furthermore, each element of the group $G$ generated by $\l_1,\dots,\l_n$ is an eigenvalue.
\end{enumerate}
\end{enumerate}
\end{theorem}

We are not aware of any non-onto symbols of isometric composition operators on $\BigD^n$. Thus, we propose the following

\begin{conjecture}Let $\f\in H(\BigD^n)$ be the symbol of an isometric composition operator on $\cal{B}$.
\begin{enumerate}
\item[{\normalfont{(a)}}] If $\varphi \not\in \Aut(\BigD^n)$, then
$\sigma(C_\varphi)=\overline{\BigD}$.

\item[{\normalfont{(b)}}] If $\varphi \in \Aut(\BigD^n)$, so that
$\varphi(z) =
(\lambda_1z_{\tau(1)},\dots,\lambda_nz_{\tau(n)})$ where
$\modu{\lambda_j} = 1$ for $j=1,\dots,n$ and $\tau \in S_n$, then
$$\sigma(C_\varphi) =
\begin{cases}
\partial\BigD,& \text{if } \mathrm{ord}(\lambda_j) = \infty \text{ for
some } $j$,\\
\G,& \text{if } \mathrm{ord}(\lambda_j)<\infty
\text{ for all } $j$,
\end{cases}$$
where $\G$ is the finite cyclic group generated
by $\lambda_1,\dots,\lambda_n$ and by the $m^{\rm{th}}$ roots of unity, with $m=\mathrm{ord}(\tau)$.
\end{enumerate}
\end{conjecture}

From Theorem~\ref{spec-poly} and Theorem \ref{cmain}, which gives a
complete characterization of the symbols of the isometric
composition operators on the Bloch space of the unit disk, we deduce
a the classification of the spectrum in the one-dimensional
case found in \cite[Theorem 5.1]{AC}.

\begin{theorem} Let $\varphi$ be the symbol of an isometric
composition operator on the Bloch space of the unit disk.
\begin{enumerate}
\item[{\normalfont{(a)}}] If $\varphi$ is not a rotation, then $\sigma(C_\varphi) =
\overline{\BigD}$.

\item[{\normalfont{(b)}}] If $\varphi(z) = \lambda z$ with $\modu{\lambda} = 1$, then
$$\sigma(C_\varphi) =
\begin{cases}
\partial\BigD,& \text{if } \mathrm{ord}(\lambda) = \infty,\\
G,& \text{if } \mathrm{ord}(\lambda)<\infty,
\end{cases}$$ where $G$ is the cyclic
group generated by $\lambda$.
\end{enumerate}
\end{theorem}

\begin{proof} The proof is the direct application of Theorem \ref{spec-poly} to
the case of the unit disk, since in this case the symbols of the
isometric composition operators are onto (see part (c) of
Theorem~\ref{cmain}).
\end{proof}

\section{Open Questions}

We conclude the paper with the following open questions.
\begin{enumerate}
\item[(1)] In the case when $D$ is the unit disk or the unit ball, it is
well known (see \cite{Z} (Theorem~5.1.7) and \cite{Z2}
(Theorem~3.9)) that for any $z,w\in D$, $$\r(z,w)=\sup\{|f(z)-f(w)|:
f\in\cal{B}, \b_f\le 1\}.$$ On the other hand, by the
characterization of the semi-norm as a Lipschitz number, it follows
immediately that for any bounded homogeneous domain $D$ and $z,w\in
D$,  $$\sup\{|f(z)-f(w)|: f\in\cal{B}, \b_f\le 1\}\le \r(z,w).$$ For
which domains $D$ does the opposite inequality hold?\\
\\
If $D$ is such
a domain and $\f\in H(D)$, we obtain a sharper lower estimate of the norm of $C_\f$. Indeed, using a normal family argument, we
can find $f\in \cal{B}$ such that $\b_f=1$ and
$\r(\f(0),0)=|f(\f(0))-f(0)|$. Letting $g(z)=f(z)-f(0)$, we obtain a
Bloch function $g$ fixing 0 with norm 1 and such that
$\r(\f(0),0)=|g(\f(0))|$. Thus $\|C_\f(g)\|_{\cal{B}}\ge
|g(\f(0)|=\r(\f(0),0).$ Since as observed in Theorem~\ref{bddcomp},
$\|C_\f\|\ge 1$, we obtain $\|C_\f\|\ge \max\{1, \r(\f(0),0)\}$. In
the case of the unit ball this is the lower estimate in
Corollary~\ref{lowest}.\\
\\
    At the end of Section 5, we saw that if $D$ is a bounded homogeneous domain having $D$ as a factor, then it is possible to construct non-trivial examples of symbols of isometric composition operators.

\item[(2)] Are there any non-trivial examples of isometric composition
operators on the Bloch space for domains $D$ which do not contain the disk as a factor?
In particular, are there any symbols of isometric composition operators on the Bloch space for the unit ball
that are not unitary transformations?\\
\\
    In the case of the polydisk, although there is a large class of non-trivial isometries $C_\f$ on $\cal{B}$, we haven't been able to find any symbols $\f$ such that $B_\f>1$.

\item[(3)] Are there any symbols $\f$ for which $C_\f$ is an isometry on
$\cal{B}$ and $B_\f>1$ if $D$ is not the unit ball?

\item[(4)] If $\f$ is a holomorphic self-map of a bounded homogeneous domain $D$ in $\C^n$ with $n>1$ and $C_\f$ is an isometry on the Bloch space of $D$, does there exist a
sequence $\{S_k\}$ in $\hbox{Aut}(D)$ such that $\{\f\circ S_k\}$
converges locally uniformly to the identity in $D$?\\
\\
    The answer in (4) is affirmative if $D=\BigD$ by Theorem~\ref{cmain}.
The following is a weaker form of a positive answer in question (4)
for the case of the polydisk (see Theorem~8 in \cite{CC3}).

\begin{theorem}\label{polyt} Let $\f\in H(\BigD^n)$ be such that $C_\f$ is an isometry on $\cal{B}$. Then there exist sequences $\{T_m^j\}_{m\in\BigN}$ $(j=1,\dots,n) $ of automorphisms of $\BigD^n$ such that $\{(\f_1\circ T_m^1,\dots,\f_n\circ T_m^n)\}$ converges uniformly on compact subsets to the identity in $\BigD^n$.
\end{theorem}

    In \cite{CC3}, the following simple example was given to show that the converse to Theorem~\ref{polyt} is false. Let $\f(z_1,z_2)=(z_1,z_1)$ for $|z_j|<1$, $j=1,2$.  Since the components of $\f$ are equal, $C_\f$ is not an isometry. On the other hand, $(\f_1(z_1,z_2),\f_2(T(z_1,z_2)))=(z_1,z_2)$, where $T(z_1,z_2)=(z_2,z_1)$.\\
\\
    In \cite{CC3}, the question on whether the components of the symbols $\f$ of the isometric composition operators for the polydisk $\f\in H(\BigD^n)$ must depend on a single distinct complex variable was raised.

\item[(5)] Are there any other holomorphic maps $\varphi$ which induce an
isometric composition operator on the Bloch space of the polydisk
besides those whose components are defined in terms of the
isometries on the Bloch space on $\BigD$?\\
\\
A negative answer to this question and Theorem~\ref{polyt} would yield a positive answer to question (4) for the case of the polydisk.

\item[(6)] Are there any isometric composition operators on the Bloch space
of the polydisk whose symbols are not onto?

\begin{definition}\label{algcon} \cite{CM} A functional Banach space $\cal{Y}$ on a set $X$ is called {\bf algebraically consistent} if the only non-zero bounded linear functionals $k$ on $\cal{Y}$ such that $k(fg)=k(f)k(g)$ whenever $f,g,fg\in \cal{Y}$ are point evaluations (i.e. functionals of the form $K_x(f)=f(x)$ for each $f\in \cal{Y}$, where $x\in X$).\end{definition}

    In Theorem~1.6 of \cite{CM} it is shown that if $\cal{Y}$ is an algebraically consistent functional Banach space consisting of analytic functions on a set $X$ whose interior is connected and dense in $X$ and such that the elements of $\cal{Y}$ are analytic in the interior of $X$ and continuous on $X$, and if $\f$ is an analytic self-map of the interior of $X$ such that $C_\f$ is bounded and invertible on $\cal{Y}$, then $\f$ is one-to-one and onto and $C_\f^{-1}=C_{\f^{-1}}$.

\item[(7)] Is the Bloch space of the polydisk algebraically consistent?\\
\\
In the case of an affirmative answer, it would follow
that if $C_\f$ is an isometry on the Bloch space of the polydisk with $\varphi \not\in \Aut(\BigD^n)$, then $C_\f$ is not invertible (so that $\sigma(C_\varphi) =
\overline{\BigD}$), without assuming that $\varphi$ is onto.
\end{enumerate}

\bibliographystyle{amsplain}
\bibliography{references.bib}
\end{document}